%% file: CD_Tiling_MinDegree_RefEditSubmission_0322_.tex
\documentclass[11pt]{article}
\usepackage{fullpage}
\usepackage{amsmath, amsfonts,amssymb,amsthm}
\usepackage{graphicx}
\usepackage{color}
\usepackage{epsf}
\usepackage{verbatim}
\usepackage{enumitem}
\usepackage[labelformat=empty]{subfig}

\newtheorem{theorem}{Theorem}[section]
\newtheorem{lemma}[theorem]{Lemma}
\newtheorem{proposition}[theorem]{Proposition}

\newtheorem{claim}[theorem]{Claim}
\newtheorem{definition}[theorem]{Definition}

\newcommand{\floor}[1]{\left\lfloor#1\right\rfloor}
\newcommand{\ceiling}[1]{\left\lceil#1\right\rceil}
\newcommand{\half}{\frac{1}{2}}

\title{A note on bipartite graph tiling}
\author{Andrzej Czygrinow\thanks{Research of this author is supported in part by NSA grant H98230-08-1-0046}, Louis DeBiasio \\School of Mathematical and Statistical Sciences\\Arizona State University\\Tempe, AZ 85287}

\begin{document}
\maketitle

\begin{abstract}
Bipartite graph tiling was studied by Zhao \cite{Z} who gave the best possible minimum degree conditions for a balanced bipartite graph on $2ms$ vertices to contain $m$ vertex disjoint copies of $K_{s,s}$. Let $s<t$ be fixed positive integers. Hladk\'y and Schacht \cite{HS} gave minimum degree conditions for a balanced bipartite graph on $2m(s+t)$ vertices to contain $m$ vertex disjoint copies of $K_{s,t}$.  Their results were best possible, except in the case when $m$ is odd and $t> 2s+1$.  We give the best possible minimum degree condition in this case.
\end{abstract}

\section{Introduction}

If $G$ is a graph on $n=sm$ vertices, $H$ is a graph on $s$ vertices and $G$ contains $m$ vertex disjoint copies of $H$, then we say $G$ can be \emph{tiled} with $H$.  In this language, we state the seminal result of Hajnal and Szemer\'edi.

\begin{theorem}[Hajnal-Szemer\'edi \cite{HSz}]\label{hsz}
Let $G$ be a graph on $n=sm$ vertices.  If $\delta(G)\geq (s-1)m$, then $G$ can be tiled with $K_s$.
\end{theorem}

For tiling with general $H$, results of Alon and Yuster \cite{Alon} and Koml\'{o}s, S\'{a}rk\"{o}zy, and  Szemer\'edi \cite{KSS} gave sufficient conditions on the minimum degree of a graph $G$ such that $G$ can be tiled with $H$.  Specifically, in \cite{KSS}, it is shown that if $G$ is a graph on $n$ vertices with minimum degree at least $\left (1 - 1/\chi(H)\right)n +K$ for a constant $K$ that only depends on $H$, then $G$ can be tiled with $H$.  A more delicate minimum degree condition that involves the so-called critical chromatic number of $H$ was conjectured by Koml\'os and solved by Shokoufandeh and Zhao \cite{Zh}.  Finally, K\"{u}hn and Osthus \cite{KO} determined exactly when the critical chromatic number or chromatic number is the appropriate parameter and thus settled the problem (for large graphs).

In this paper we study the tiling problem in bipartite graphs.  Denote a bipartite graph $G$ with partition sets $U$ and $V$ by $G[U, V]$.  We say $G[U,V]$ is \emph{balanced} if $|U|=|V|$.  Zhao proved the following Hajnal-Szemer\'edi type result for bipartite graphs.

\begin{theorem}[Zhao \cite{Z}]
For each $s\geq 2$, there exists $m_0$ such that the following holds for all $m\geq m_0$.  If $G$ is a balanced bipartite graph on $2n=2ms$ vertices with 
$$\delta(G)\geq \left\lbrace \begin{array}{ll} \frac{n}{2}+s-1   & \text{ if } m \text{ is even } \\
              \frac {n+3s}{2}-2 & \text{ if } m \text{ is odd, } \end{array} \right. $$
then $G$ can be tiled with $K_{s,s}$.
\end{theorem}
Zhao proved that this minimum degree condition was tight.

\begin{proposition}[Zhao \cite{Z}]
Let $s\geq 2$, and $n=ms\geq 64s^2$.  There exists a balanced bipartite graph, $G$, on $2n$ vertices with
$$\delta(G)= \left\lbrace \begin{array}{ll} \frac{n}{2}+s-2   & \text{ if } m \text{ is even } \\
              \frac {n+3s}{2}-3 & \text{ if } m \text{ is odd } \end{array} \right. $$
such that $G$ cannot be tiled with $K_{s,s}$.
\end{proposition}

Hladk\'{y} and Schacht extended Zhao's result as follows.

\begin{theorem}[Hladk\'{y}-Schacht \cite{HS}]\label{HS ub}
Let $1\leq s<t$ be fixed integers.  There exists $m_0$ such that the following holds for all $m\geq m_0$. If $G$ is a balanced bipartite graph on $2n=2m(s+t)$ vertices with $$\delta(G)\geq \left\lbrace \begin{array}{ll} \frac{n}{2}+s-1   & \text{ if } m \text{ is even } \\
               \frac {n+t+s}{2}-1 & \text{ if } m \text{ is odd,} \end{array} \right. $$
then $G$ can be tiled with $K_{s,t}$.
\end{theorem}
They proved that this minimum degree condition was tight in all cases except when $m$ is odd and $t> 2s+1$.  Note that since we are dealing with balanced bipartite graphs, in any tiling of $G[U,V]$ with $K_{s,t}$ there must be an equal number of copies of $K_{s,t}$ with $s$ vertices in $U$ as copies of $K_{s,t}$ with $t$ vertices in $U$.  This explains why the authors \cite{HS} suppose $2n=2m(s+t)$ instead of $2n=m(s+t)$.  

\begin{proposition}[Hladk\'{y}-Schacht \cite{HS}]
Let $1\leq s<t$ be fixed integers.  There exists $m_0$ such that the following holds for all $m\geq m_0$. There exists a balanced bipartite graph, $G$, on $2n=2m(s+t)$ vertices with
$$\delta(G)= \left\lbrace \begin{array}{ll} \frac{n}{2}+s-2   & \text{ if } m \text{ is even } \\
              \frac {n+t+s}{2}-2 & \text{ if } m \text{ is odd and } t\leq 2s+1 \end{array} \right. $$
such that $G$ cannot be tiled with $K_{s,t}$.
\end{proposition}

Our objective is to give the tight minimum degree condition in the final remaining case, when $m$ is odd and $t> 2s+1$. We will do this in two parts.  First in Section \ref{Extremal} we prove that when $m$ is odd and $t\geq 2s+1$, the following minimum degree condition is sufficient.

\begin{theorem}\label{main}
Let $1\leq s<t$ be fixed integers with $2s+1\leq t$.  There exists $m_0$ such that the following holds for all odd $m$ with $m\geq m_0$. If $G$ is a balanced bipartite graph on $2n=2m(s+t)$ vertices with $$\delta(G)\geq \frac {n+3s}{2}-1,$$ then $G$ can be tiled with $K_{s,t}$.
\end{theorem}

Then in Section \ref{lower bound} we prove that the minimum degree condition in Theorem \ref{main} is tight.

\begin{proposition}\label{counterexample}
Let $1\leq s<t$ be fixed integers with $2s+1\leq t$.  There exists $m_0$ such that the following holds for all odd $m$ with $m\geq m_0$. There exists a balanced bipartite graph, $G$, on $2n=2m(s+t)$ vertices with 
$$\delta(G)= \left\lbrace \begin{array}{ll}\frac{n+3s}{2}-\frac{3}{2}   & \text{ if } t \text{ is odd } \\
              \frac{n+3s}{2}-2 & \text{ if } t \text{ is even } \end{array} \right. $$
such that $G$ cannot be tiled with $K_{s,t}$.
\end{proposition}

Let $m=2k+1$ for some $k\in\mathbb{N}$ and let $n=m(s+t)$.  We note that when $t=2s+1$, $\frac{n+3s}{2}-1=(k+1)(s+t)-\frac{3}{2}$ and $\frac{n+t+s}{2}-1=(k+1)(s+t)-1$.  So the value for the lower bound in Theorem \ref{main} is smaller than the value for the lower bound in Theorem \ref{HS ub} when $t=2s+1$, but since $\delta(G)$ only takes integer values the minimum degree condition in Theorem \ref{main} is not an improvement until $t>2s+1$.

\section{Proof of Theorem \ref{main}} 

For disjoint sets $A,B\subseteq V(G)$, we define $e(A,B)$ to be the number of edges with one end in $A$ and the other end in $B$ and for $v\in V(G)\setminus A$ we write $\deg(v,A)$ instead of $e(\{v\},A)$.  Also, $d(A,B)=\frac{e(A,B)}{|A||B|}$, $\delta(A,B)=\min\{\deg(v,B):v\in A\}$ and $\Delta(A,B)=\max\{\deg(v,B):v\in A\}$.  An \emph{$h$-star from $A$ to $B$}, is a copy of $K_{1,h}$ with the vertex of degree $h$, \emph{the center}, in $A$ and the vertices of degree $1$, \emph{the leaves}, in $B$.

The following theorem appears in \cite{Z}.

\begin{theorem}[Zhao \cite{Z}]\label{stability}
For every $\alpha>0$ and every positive integer $r$, there exist $\beta>0$ and positive integer $m_1$ such that the following holds for all $n=mr$ with $m\geq m_1$.  Given a bipartite graph $G[U,V]$ with $|U|=|V|=n$, if $\delta(G)\geq (\half-\beta)n$, then either $G$ can be tiled with $K_{r,r}$, or there exist \begin{equation}\label{extremalcondition} U_1'\subseteq U,~ V_2'\subseteq V,~\text{ such that } |U_1'|=|V_2'|=\floor{n/2},~ d(U_1',V_2')\leq \alpha.\end{equation}
\end{theorem}

If a balanced bipartite graph $G[U,V]$ on $2n$ vertices with $n$ divisible by $r$ satisfies (\ref{extremalcondition}), we say $G$ is \emph{extremal} with parameter $\alpha$.  In this case we set $U_2':=U\setminus U_1'$ and $V_1':=V\setminus V_2'$.

If we replace $r$ with $s+t$ in Theorem \ref{stability}, we see that either $G$ can be tiled with $K_{s+t,s+t}$ or else we are in the extremal case.  If it is the case that $G$ can be tiled with $K_{s+t,s+t}$, we split each copy of $K_{s+t,s+t}$ into two copies of $K_{s,t}$ to give the desired tiling.  So we must only deal with the extremal case.

\subsection{Pre-processing}\label{preprocess}

\begin{claim}\label{diagonals}
Let $0<\alpha\ll 1$, $r\in \mathbb{N}$ and let $m_1\in \mathbb{N}$ be given by Theorem \ref{stability}.  Let $m\geq m_1$ and suppose that $G[U,V]$ is a balanced bipartite graph on $2n=2mr$ vertices such that $\delta(G)= \frac{n}{2}+C$, where $0\leq C\leq 3r/2$.  Suppose further that the deletion of any edge of $G$ will cause the resulting graph to have minimum degree less than $\frac{n}{2} +C$.  If $G$ is extremal with parameter $\alpha$, then $d(U_2', V_1')\leq 5\sqrt{\alpha}$.
\end{claim}

\begin{proof}
Let $\gamma:=5\sqrt{\alpha}$ and suppose $d(U_2',V_1')>\gamma$.  Let $X'=\{u\in U_2':\deg(u,V_2')<(1-\sqrt{\alpha})\frac{n}{2}\}$, $Y'=\{v\in V_1':\deg(v,U_1')<(1-\sqrt{\alpha})\frac{n}{2}\}$. Since $e(U_1',V_2')\leq \alpha \frac{n^2}{4}$ and $e(U_1',V)\geq |U_1'|\frac{n}{2}$, we have $e(U_1', V_1')\geq |U_1'|\frac{n}{2}-\alpha\frac{n^2}{4}$.  Thus we can bound the non-edges between $U_1'$ and $V_1'$, $$\sqrt{\alpha}\frac{n}{2}|Y'|\leq \bar{e}(U_1',V_1')\leq \alpha\frac{n^2}{4},$$ which gives $|Y'|\leq\sqrt{\alpha}\frac{n}{2}$.  Similarly we have $|X'|\leq \sqrt{\alpha}\frac{n}{2}$. Let $U_2''=U_2'\setminus X'$ and $V_1''=V_1'\setminus Y'$.  Since $d(U_2', V_1')>\gamma$, we have
\begin{equation}\label{doubleprime}
e(U_2'',V_1'')\geq \gamma\frac{n^2}{4}-2\sqrt{\alpha}\frac{n^2}{4}=3\sqrt{\alpha}\frac{n^2}{4}.
\end{equation}

Let $X''=\{u\in U_2'':\deg(u,V_1'')\geq \sqrt{\alpha}\frac{n}{2}+C+1\}$ and $Y''=\{v\in V_1'':\deg(v,U_2'')\geq \sqrt{\alpha}\frac{n}{2}+C+1\}$.  If there is an edge $uv\in E(X'', Y'')$, then $\deg(u),\deg(y)\geq\frac{n}{2}+C+1$ which contradicts the edge minimality of $G$, so suppose $e(X'', Y'')=0$. Finally, by \eqref{doubleprime} we have 
$$3\sqrt{\alpha}\frac{n^2}{4}\leq e(U_2'',V_1'')\leq e(X'',Y'')+e(U_2''\setminus X'',V_1'')+e(V_1''\setminus Y'', U_2'')\leq 0+2(\sqrt{\alpha}\frac{n}{2}+C)\frac{n}{2},$$
which is a contradiction, since $n$ is sufficiently large.

\end{proof}

Let $1\leq s<t$ be integers so that $2s+1\leq t$, and let $0<\alpha\ll 1$ (setting $\alpha:=\left(\frac{1}{32t(s+t)}\right)^{3}$ is small enough).  Let $G[U,V]$ be a balanced bipartite graph on $2n=2m(s+t)$ vertices, where $m=2k+1$ and $k$ is a sufficiently large integer with respect to $(\frac{\alpha}{5})^2$.  Suppose that $G$ is extremal with parameter $(\frac{\alpha}{5})^2$ and edge-minimal with respect to the condition $\delta(G)\geq \frac{n+3s}{2}-1$.  By Claim \ref{diagonals} we have $d(U_i',V_{3-i}')\leq \alpha$ for $i=1,2$.  Then for $i=1,2$, we define
\begin{align*}
&U_i=\{u\in U:\deg(u,V_{3-i}')<\alpha^{\frac{1}{3}}\frac{n}{2}\},~ V_i=\{v\in V:\deg(v,U_{3-i}')<\alpha^{\frac{1}{3}}\frac{n}{2}\},\\
&U_0=U-U_1-U_2, \text{ and } V_0=V-V_1-V_2.
\end{align*}
As a consequence of these definitions, we have the following.
\begin{claim}\label{bounds} For $i=1,2$
\begin{align*}
\emph{(i)}&~ (1-\alpha^{2/3})\frac{n}{2}\leq |U_i|,|V_i|\leq (1+\alpha^{2/3})\frac{n}{2},~~~  \emph{(ii)}~|U_0|,|V_0|\leq \alpha^{2/3}n,\\
\emph{(iii)}&~(1-2\alpha^{1/3})\frac{n}{2}<\delta(U_i,V_i),\delta(V_i,U_i),~~~ \emph{(iv)}~(\alpha^{1/3}-\alpha^{2/3})\frac{n}{2}\leq \delta(U_0,V_i),\delta(V_0,U_i),\\ \emph{(v)}&~\Delta(U_i,V_{3-i}), \Delta(V_{3-i},U_i)\leq \alpha^{1/3}n
\end{align*}
\end{claim}

\begin{proof}
A proof of (i)-(iv) can be found in \cite{Z} and was also used in \cite{HS}.  So we prove (v) here.  

Let $i\in \{1,2\}$ and note that 
\begin{equation}\label{eq1}
|U_i'\setminus U_i|\alpha^{1/3}\frac{n}{2}\leq e(U_i'\setminus U_i, V_{3-i}')\leq e(U_i', V_{3-i}')\leq \alpha\frac{n^2}{4}
\end{equation}
and 
\begin{equation}\label{eq2}
|V_i'\setminus V_i|\alpha^{1/3}\frac{n}{2}\leq e(V_i'\setminus V_i, U_{3-i}')\leq e(V_i', U_{3-i}')\leq \alpha\frac{n^2}{4}.
\end{equation}
Then (\ref{eq1}) and (\ref{eq2}) imply 
\begin{equation}\label{eq3}
|U_i'\setminus U_i|, |V_i'\setminus V_i|\leq \alpha^{2/3}\frac{n}{2},
\end{equation}
which gives $\Delta(U_i, V_{3-i})\leq \Delta(U_i, V_{3-i}')+|V_{3-i}\setminus V_{3-i}'|\leq \Delta(U_i, V_{3-i}')+|V_{i}'\setminus V_{i}|\leq \alpha^{1/3}n$ and $\Delta(V_i, U_{3-i})\leq \Delta(V_i, U_{3-i}')+|U_{3-i}\setminus U_{3-i}'|\leq \Delta(V_i, U_{3-i}')+|U_{i}'\setminus U_{i}|\leq \alpha^{1/3}n$.

\end{proof}

We need to define some new sets which were not specified in \cite{Z}.  
\begin{definition}\label{tildehat}
For $i= 1,2$, let 
\begin{align*}
&\tilde{U_i}=\{u\in U_i:\deg(u,V_{3-i})\geq s\},~ \tilde{V_i}=\{v\in V_i:\deg(v,U_{3-i})\geq s\},\\
&\hat{U_i}=U_i\setminus \tilde{U_i}, \text{ and } \hat{V_i}=V_i\setminus \tilde{V_i}. 
\end{align*}
\end{definition}

Note that the following inequalities are satisfied:
\begin{align}
\delta(\hat{U_1},V_0)+\delta(\hat{U_2},V_0)&\geq n+3s-2-(|V_1|+s-1)-(|V_2|+s-1)=|V_0|+s ~~\text{and} \label{V_0}\\
\delta(\hat{V_1},U_0)+\delta(\hat{V_2},U_0)&\geq n+3s-2-(|U_1|+s-1)-(|U_2|+s-1)=|U_0|+s. \label{U_0}
\end{align}

\subsection{Preliminary Claims}

The following useful lemma appears in \cite{Z}.

\begin{lemma}[Zhao \cite{Z}, Fact 5.3]\label{lem:Zhao}
Let $F[A, B]$ be a bipartite graph with $\delta:=\delta(A,B)$ and $\Delta:=\Delta(B,A)$ 
Then $F$ contains $f_h$ vertex disjoint $h$-stars from $A$ to $B$, and $g_h$ vertex disjoint $h$-stars from $B$ to $A$ (the stars from $A$ to $B$ and those from $B$ to $A$ need not be disjoint), where
\begin{align*}
f_h\geq\frac{(\delta-h+1)|A|}{h\Delta+\delta-h+1}, ~~~ g_h\geq\frac{\delta|A|-(h-1)|B|}{\Delta+h\delta-h+1}.
\end{align*}

\end{lemma}

We now prove three claims that we will need in the main proof.

\begin{claim}\label{TildeStars}
Let $i\in\{1,2\}$ and $\{A,B\}=\{U_i,V_{3-i}\}$.  Let $0\leq c\leq \alpha^{1/3}n$, $B_0\subseteq B$ and $A_0=\{v\in A:\deg(v,B_0)\geq s+c\}$.
If $|A_0|\geq\frac{n}{4}$ then there is a set $\mathcal{S}_A$ of at least $\frac{c+1}{8s\alpha^{1/3}}$ vertex disjoint $s$-stars from $A_0$ to $B_0$.
\end{claim}

\begin{proof}
Let $\mathcal{S}_A$ be a maximum set of vertex disjoint $s$-stars from $A_0$ to $B_0$ and let $f_s=|\mathcal{S}_A|$.  We apply Lemma \ref{lem:Zhao} to the graph $G[A_0,B_0]$.   Recall, by Claim \ref{bounds}, that $\Delta(B,A)\leq \alpha^{1/3} n$.  Then
\begin{align*}
f_s\geq \frac{(c+1)|A_0|}{s\alpha^{1/3}n+c+1}\geq\frac{(c+1)\frac{n}{4}}{2s\alpha^{1/3}n}=\frac{c+1}{8s\alpha^{1/3}}.
\end{align*}

\end{proof}

Note that since $n=(2k+1)(s+t)$, we can write $\delta(G)\geq \frac{n+3s}{2}-1=k(s+t)+2s+\frac{t}{2}-1$.

\begin{claim} \label{Stars}Let $i\in\{1,2\}$ and $\{A,B\}=\{U_i,V_{3-i}\}$.  Let $|A|=k(s+t)+z$ and $|B|=k(s+t)+y$.  Suppose $y\geq z$ and $y\geq \frac{t+1}{2}$.  Then there is a set $\mathcal{S}_B$ of $y$ vertex disjoint $s$-stars with centers $C_B\subseteq B$ and leaves $L_A\subseteq A$.  Furthermore if $z\geq 1$, then there is a set $\mathcal{S}_A$ of $z$ vertex disjoint $s$-stars from $A\setminus L_A$ to $B\setminus C_B$.
\end{claim}

\begin{proof}
Let $\beta:=32s\alpha^{1/3}$ and recall that by the choice of $\alpha$ we have $\frac{1}{t}\gg \beta\gg 2\alpha^{1/3}$.  
We show that the desired set $\mathcal{S}_B$ exists by applying Lemma \ref{lem:Zhao} to the graph $G[A,B]$.  We have $\delta(A,B)\geq k(s+t)+2s+\frac{t}{2}-1-(n-|B|)=y+s-\frac{t}{2}-1$ and $\Delta(B,A)\leq \alpha^{1/3}n$ by Claim \ref{bounds}.  Let $g_s=|\mathcal{S}_B|$, then
\begin{align*}
g_s&\geq \frac{(y-\frac{t}{2}+s-1)(k(s+t)+z)-(s-1)(k(s+t)+z+y-z)}{\alpha^{1/3}n+s(y-\frac{t}{2}+s-1)-s+1}\\
&=\frac{(y-\frac{t}{2})(k(s+t)+z)-(s-1)(y-z)}{\alpha^{1/3}n+s(y-\frac{t}{2})+s^2-2s+1}\\
&\geq \frac{(y-\frac{t}{2})\frac{n}{3}}{2\alpha^{1/3}n}~~~~~~(\text{since } y\leq \alpha^{2/3}\frac{n}{2} \text{ and } -\alpha^{2/3}\frac{n}{2}\leq z, \text{ by Claim \ref{bounds}} )\\
&\geq y ~~~~~~(\text{since } y\geq \frac{t+1}{2} \text{ and } \alpha\ll 1).
\end{align*}
Thus the desired set $\mathcal{S}_B$ exists.

Suppose $z\geq 1$.  Let $c:=\frac{1}{2}y$ if $y\geq 1/\beta$, and let $c:=0$ if $y<1/\beta$.  Let $B_0=B\setminus C_B$ and $A_0=\{v\in A\setminus L_A|\deg(v,B_0)\geq s+c\}$ and $\bar{A}=(A\setminus L_A)\setminus A_0$.  Suppose that $|\bar{A}|\geq\frac{n}{16}$.  Then there exists $u\in C_B$ such that if $y<1/\beta$,
\begin{align*}
\deg(u,A)\geq\frac{e(\bar{A},C_B)}{|C_B|}\geq\frac{\left(y-\frac{t}{2}+s-1-(s-1)\right)\frac{n}{16}}{y}
=\frac{\left(y-\frac{t}{2}\right)\frac{n}{16}}{y}
> \frac{\beta n}{32}
\geq \alpha^{1/3}n
\end{align*}
and if $y\geq 1/\beta$,
\begin{align*}
\deg(u,A)\geq\frac{e(\bar{A},C_B)}{|C_B|}>\frac{\left(y-\frac{t}{2}+s-1-(s+\half y)\right)\frac{n}{16}}{y}
=\frac{\left(\frac{y}{2}-\frac{t}{2}-1\right)\frac{n}{16}}{y}
>\frac{n}{64}
\geq\alpha^{1/3}n,
\end{align*}
each contradicting Claim \ref{bounds}.  So $|\bar{A}|<\frac{n}{16}$ and thus $|A_0|\geq |A|-|L_A|-\frac{n}{16}\geq k(s+t)-s\alpha^{2/3}\frac{n}{2}-\frac{n}{16}\geq\frac{n}{4}$. Now let $\mathcal{S}_A$ be a maximum set of disjoint $s$-stars from $A_0$ to $B_0$ and let $f_s=|\mathcal{S}_A|$. By Lemma \ref{TildeStars} we have $f_s\geq \frac{c+1}{8s\alpha^{1/3}}$.  Recall that $1\leq z\leq y$.  If $y\geq 1/\beta$, then $f_s\geq \frac{y}{16s\alpha^{1/3}}\geq z$ and if $y<1/\beta$, then $f_s\geq \frac{1}{8s\alpha^{1/3}}\geq \frac{1}{\beta}\geq z$.  So the desired set $\mathcal{S}_A$ exists.

\end{proof}

\begin{claim}\label{K_1}
Suppose $|U_0|,|V_0|\geq s$.  If $|\hat{U_1}|\geq\frac{n}{8}$ and $|\hat{U_2}|\geq\frac{n}{8}$ (see Definition \ref{tildehat}), then there is a $K_{s,t}=:K^1$ with $s$ vertices in $V_0$, $\ceiling{t/2}$ vertices in $U_1$ and $\floor{t/2}$ vertices in $U_2$.  Likewise, if $|\hat{V_1}|\geq\frac{n}{8}$ and $|\hat{V_2}|\geq\frac{n}{8}$ then there is a $K_{s,t}=:K^2$ with $s$ vertices in $U_0$, $\ceiling{t/2}$ vertices in $V_1$ and $\floor{t/2}$ vertices in $V_2$.
\end{claim}

\begin{proof}

Without loss of generality we will only prove the first statement.  Let $$\ell:=s\binom{|U_2|}{\floor{t/2}}/\binom{\ceiling{(\alpha^{1/3}-\alpha^{2/3})n/2}}{\floor{t/2}}$$
and recall that $|U_1|, |U_2|\leq (1+\alpha^{2/3})\frac{n}{2}$ by Claim \ref{bounds}. Thus we have
\begin{equation}\label{ell}
\ell\leq s\left(\frac{|U_2|}{(\alpha^{1/3}-\alpha^{2/3})\frac{n}{2}-\floor{t/2}}\right)^{\floor{t/2}}\leq  s\left(\frac{(1+\alpha^{2/3})\frac{n}{2}}{(\alpha^{1/3}-\alpha^{2/3})\frac{n}{3}}\right)^{\floor{t/2}}\leq s\left(\frac{3(1+\alpha^{2/3})}{2(\alpha^{1/3}-\alpha^{2/3})}\right)^{\floor{t/2}}  .\end{equation}

\noindent
\textbf{Case 1. } $|V_0|\geq  \ell \binom{|U_1|}{\ceiling{t/2}}/\binom{\ceiling{(\alpha^{1/3}-\alpha^{2/3})n/2}}{\ceiling{t/2}}$.
Recall that $\delta(V_0, U_i) \geq (\alpha^{1/3} - \alpha^{2/3})n/2$ for $i=1,2$ by Claim \ref{bounds} and suppose that there is no $K_{\ceiling{t/2},\ell}$ with $\ceiling{t/2}$ vertices in $U_1$ and $\ell$ vertices in $V_0$.  We count the $\ceiling{t/2}$-stars from $V_0$ to $U_1$ in two ways which gives $$|V_0| \binom{\ceiling{(\alpha^{1/3} - \alpha^{2/3})n/2}}{\ceiling{t/2}} < \ell \binom{|U_1|}{\ceiling{t/2}}$$
contradicting the lower bound for $|V_0|$. Consequently there is a complete bipartite graph $K'=K_{\ceiling{t/2},\ell}$ with $\ceiling{t/2}$ vertices in $U_1$ and $\ell$ vertices in $V_0$. If there is no $K_{\floor{t/2},s}$ with $s$ vertices in $V(K') \cap V_0$ and $\floor{t/2}$ vertices in $U_2$, then a similar counting argument gives $$\ell \binom{\ceiling{(\alpha^{1/3} - \alpha^{2/3})n/2}}{\floor{t/2}} < s\binom{|U_2|}{\floor{t/2}}$$
contradicting the definition of $\ell$.

\noindent
\textbf{Case 2.} $|V_0|<  \ell \binom{|U_1|}{\ceiling{t/2}}/\binom{\ceiling{(\alpha^{1/3}-\alpha^{2/3})n/2}}{\ceiling{t/2}}$.  
By \eqref{ell}, we have
$$|V_0|< 
\ell\left(\frac{3(1+\alpha^{2/3})}{2(\alpha^{1/3}-\alpha^{2/3})}\right)^{\ceiling{t/2}}\leq s\left(\frac{3(1+\alpha^{2/3})}{2(\alpha^{1/3}-\alpha^{2/3})}\right)^{t}.$$  Let $p := \delta(\hat{U_1}, V_0)$, and note that $p\geq s$ by (\ref{V_0}).  We claim that there is a complete bipartite graph $K':= K_{\ceiling{t/2},p}$ with $\ceiling{t/2}$ vertices in $\hat{U_1}$ and $p$ vertices in $V_0$. Let $c$ be the number of $p$-stars with centers in $\hat{U_1}$ and leaves in $V_0$.  We have $c\geq |\hat{U_1}|\geq \frac{n}{8}$ and if no $p$-subset of $V_0$ is in $\ceiling{t/2}$ of such stars, i.e. $K'$ does not exist, we have $c\leq(\ceiling{t/2}-1)\binom{|V_0|}{p}$ which contradicts the fact that $|V_0|$ is $O(1)$ and $n$ is sufficiently large (with respect to $\alpha$, $t$, and consequently $|V_0|$). From (\ref{V_0}) we have $\delta(\hat{U_2},V_0)\geq|V_0|-p+s$, so every vertex $u \in  \hat{U_2}$ has at least $s$ neighbors in $V(K')\cap V_0$. Repeating the argument above by counting $s$-stars with centers in $\hat{U_2}$ and leaves in $V(K')\cap V_0$ gives $K'':=K_{s,\floor{t/2}}$.  Now choose $K^1\subseteq K'\cup K''$ having the property that $|V_0 \cap V(K^1)|=s$, $|U_1\cap V(K^1)|=\ceiling{t/2}$, and $|U_2\cap V(K^1)|=\floor{t/2}$ as desired.
\end{proof}

\subsection{Extremal Case}
\label{Extremal}

Recall that $t\geq 2s+1$, $n=(2k+1)(s+t)$ for some sufficiently large $k\in \mathbb{N}$, and $\delta(G)\geq \frac{n+3s}{2}-1=k(s+t)+2s+\frac{t}{2}-1$.  We start with the partition given in Section \ref{preprocess} and we call $U_0$ and $V_0$ the \emph{exceptional} sets. Let $i\in \{1,2\}$. We will attempt to update the partition by moving a constant number (depending only on $t$) of \emph{special} vertices between $U_1$ and $U_2$, denote them by $X$, and \emph{special} vertices between $V_1$ and $V_2$, denote them by $Y$, as well as partitioning the exceptional sets as $U_0=U_0^1\cup U_0^2$ and $V_0=V_0^1\cup V_0^2$.  Let $U_1^*$, $U_2^*$, $V_1^*$ and $V_2^*$ be the resulting sets after moving the special vertices. 
Our goal is to obtain two graphs, $G_1:=G[U_1^*\cup U_0^1, V_1^*\cup V_0^1]$ and $G_2:=[U_2^*\cup U_0^2, V_2^*\cup V_0^2]$ so that $G_1$ satisfies $$|U_1^*\cup U_0^1|=\ell_1(s+t)+as+bt, |V_1^*\cup V_0^1|=\ell_1(s+t)+bs+at$$ and $G_2$ satisfies $$|U_2^*\cup U_0^2|=\ell_2(s+t)+bs+at, |V_2^*\cup V_0^2|=\ell_2(s+t)+as+bt,$$ for some nonnegative integers $a,b,\ell_1,\ell_2$.  We tile $G_1$ as follows.  We find $a$ copies of $K_{s,t}$, each with $t$ vertices in $U_1^*$, so that each special vertex in $X\cap U_1^*$ is in a unique copy (some copies may not contain any special vertex). Also, we find $b$ copies of $K_{s,t}$, each with $t$ vertices in $V_1^*$ so that each special vertex in $Y\cap V_1^*$ is in a unique copy (some copies may not contain any special vertex).  Note that we only move vertices which will make this step possible. Deleting these $a+b$ copies of $K_{s,t}$ from $G_1$ gives us a balanced bipartite graph on $2\ell_1(s+t)$ vertices.  As noted in \cite{Z} and \cite{HS}, this graph can easily be tiled:  By Claim \ref{bounds} there are at most $\alpha^{2/3}\frac{n}{2}$ exceptional vertices in $U_0^1$ (resp.~$V_0^1$), each with degree at least $(\alpha^{1/3}-\alpha^{2/3})\frac{n}{2}$ to $V_1$ (resp.~$U_1$), so they may greedily be incorporated into unique copies of $K_{s+t,s+t}$.  The remaining graph is still balanced, divisible by $s+t$, and almost complete, thus can be tiled.

So if we are able to split $G$ into graphs $G_1$ and $G_2$ as detailed above, we will conclude that $G$ can be tiled.  However, if it is not possible to carry out this goal, then we will use an alternate method which is explained in Case 2.

\begin{proof}[Proof of Theorem 1.6]
There are two main cases.

\noindent
\textbf{Case 1.} $\max\{|U_1|,|U_2|,|V_1|,|V_2|\}\geq k(s+t)+\frac{t+1}{2}$.  Without loss of generality, suppose $|U_1|=\max\{|U_1|,|U_2|,|V_1|,|V_2|\}$. 

\textbf{Case 1.1.} $|V_2\cup V_0|\geq k(s+t)+s$.  We apply Claim \ref{Stars} to $G[U_1, V_2]$ with $A=V_2$ and $B=U_1$ to obtain $|U_1|-(k(s+t)+s)$ vertex disjoint $s$-stars with centers $C_U\subseteq U_1$ and leaves in $V_2$ and a set of $\max\{0, |V_2|-(k(s+t)+s)\}$ vertex disjoint $s$-stars with centers $C_V\subseteq V_2$ and leaves in $U_1$.  We move the vertices in $C_U$ to $U_2$ and the vertices in $C_V$ to $V_1$.  If $|V_2|<k(s+t)+s$, we choose $V_0'\subseteq V_0$ so that $|(V_2\cup V_0)\setminus V_0')|=k(s+t)+s$ otherwise we set $V_0'=\emptyset$. Then $G_1:=G[U_1\setminus C_U, V_1\cup C_V\cup V_0']$ satisfies $$|U_1|-|C_U|=k(s+t)+s, |V_1|+|V_0'|+|C_V|=k(s+t)+t,$$ and $G_2:=G-G_1$ satisfies $$|U_2\cup U_0|+|C_U|=k(s+t)+t, |V_2|+|V_0\setminus V_0'|-|C_V|=k(s+t)+s.$$  Thus $G_1$ and $G_2$ can be tiled, which completes the tiling of $G$.

\textbf{Case 1.2.} $|V_2\cup V_0|<k(s+t)+s$.  

This implies $|V_1|> k(s+t)+t$.  So we apply Claim \ref{Stars} to $G[V_1,U_2]$ with $A=U_2$ and $B=V_1$ to obtain a set of $|V_1|-k(s+t)$ vertex disjoint $s$-stars with centers $C_V\subseteq V_1$ and leaves in $U_2$. Likewise we apply Claim \ref{Stars} to $G[U_1,V_2]$ with $A=V_2$ and $B=U_1$ to obtain a set of $|U_1|-k(s+t)$ vertex $s$-stars with centers $C_U\subseteq U_1$ and leaves in $V_2$.  We move the vertices in $C_U$ to $U_2$ and the vertices in $C_V$ to $V_2$.  Then $G_1:=G[U_1\setminus C_U, V_1\setminus C_V]$ satisfies $$|U_1|-|C_U|=k(s+t), |V_1|-|C_V|=k(s+t)$$ and $G_2:=G-G_1$ satisfies $$|U_2\cup U_0|+|C_U|=(k+1)(s+t), |V_2\cup V_0|+|C_V|=(k+1)(s+t).$$  Thus $G_1$ and $G_2$ can be tiled, which completes the tiling of $G$.

\noindent
\textbf{Case 2.} $\max\{|U_1|,|U_2|,|V_1|,|V_2|\}\leq k(s+t)+\frac{t}{2}$. Note that this implies $|U_0|, |V_0|\geq s$. 

\textbf{Case 2.1.} $\max\{|\tilde{U}_1|, |\tilde{U}_2|, |\tilde{V}_1|, |\tilde{V}_2|\}\geq \frac{n}{4}$ (see Definition \ref{tildehat}).  Without loss of generality we can assume $|\tilde{U}_1|=\max\{|\tilde{U}_1|, |\tilde{U}_2|, |\tilde{V}_1|, |\tilde{V}_2|\}$.    
Set $h:=\ceiling{t/(2s)}$. Since $|\tilde{U}_1|>\frac{n}{4}$ and $\frac{1}{8s\alpha^{1/3}}\geq (h-1)(s+t)$, we can apply Claim \ref{TildeStars} to $G[\tilde{U}_1, V_2]$ with $c=0$ to obtain a set of $(h-1)(s+t)$ vertex disjoint $s$-stars with centers $C_U\subseteq \tilde{U}_1$ and leaves in $V_2$. We first move the vertices in $C_U$ from $\tilde{U}_1$ to $U_2$.  Then since $$\frac{t}{2}=s\frac{t}{2s}\leq sh\leq s\frac{t+2s-1}{2s}=\frac{t}{2}+s-\frac{1}{2},$$ we can choose sets $U_0'\subseteq U_0$ with $|U_0'|=k(s+t)+\floor{t/2}-|U_1|+sh-\floor{t/2}$ and $V_0'\subseteq V_0$ with $|V_0'|=k(s+t)+\floor{t/2}-|V_1|+s+\ceiling{t/2}-sh$ so that $G_1:=G[(U_1\cup U_0')\setminus C_U, V_1\cup V_0']$ satisfies $$|U_1|+|U_0'|-|C_U|=(k-h+1)(s+t)+hs, |V_1|+|V_0'|=(k-h+1)(s+t)+ht,$$
and $G_2:=G-G_1$ satisfies $$|U_2|+|U_0\setminus U_0'|+|C_U|=k(s+t)+ht, |V_2|+|V_0\setminus V_0'|= k(s+t)+hs.$$  Thus $G_1$ and $G_2$ can be tiled, which completes the tiling of $G$.

\textbf{Case 2.2.} $\max\{|\tilde{U}_1|, |\tilde{U}_2|, |\tilde{V}_1|, |\tilde{V}_2|\}< \frac{n}{4}$.  Thus for $i=1,2$, we have $$|\hat{U_i}|, |\hat{V_i}|\geq (1-\alpha^{2/3})\frac{n}{2}-\frac{n}{4}\geq \frac{n}{8}.$$  So we may apply Claim \ref{K_1} to obtain the two special copies of $K_{s,t}$, $K^1$ and $K^2$.  Note that $|U_i\setminus V(K^1)|$, $|V_i\setminus V(K^2)|\leq k(s+t)$ for $i=1,2$.  Let $U_0'=U_0\setminus V(K^2)$ and $V_0'=V_0\setminus V(K^1)$.  We remove the graphs $K^1$ and $K^2$, then we partition the vertices $U_0'=U_0^1\cup U_0^2$ and $V_0'=V_0^1\cup V_0^2$ so that $G_1:=G[(U_1\cup U_0^1)\setminus V(K^1), (V_1\cup V_0^1)\setminus V(K^2)]$ satisfies $$|U_1|-\ceiling{t/2}+|U_0^1|=k(s+t), |V_1|-\ceiling{t/2}+|V_0^1|=k(s+t)$$ and $G_2=G-G_1-K^1-K^2$ satisfies $$|U_2|-\floor{t/2}+|U_0^2|=k(s+t), |V_2|-\floor{t/2}+|V_0^2|=k(s+t).$$  Thus $G_1$ and $G_2$ can be tiled, so along with $K^1$ and $K^2$, this completes the tiling of $G$.

\end{proof}

\section{Tightness}
\label{lower bound}

In this section we will prove Proposition \ref{counterexample}.  We will need to use the graphs $P(m,p)$, where $m,p\in\mathbb{N}$, introduced by Zhao in \cite{Z}.

\begin{lemma}\label{No K22}

For all $p\in\mathbb{N}$ there exists $m_0$ such that for all $m\in \mathbb{N}$, $m>m_0$, there exists a balanced bipartite graph, $P(m,p)$, on $2m$ vertices, so that the following hold:
\begin{enumerate}
\item $P(m,p)$ is $p$-regular

\item $P(m,p)$ does not contain a copy of $K_{2,2}$.
\end{enumerate}

\end{lemma}

\begin{proof}[Proof of Proposition \ref{counterexample}]

Let $G[U, V]$ be a balanced bipartite graph on $2n$ vertices satisfying the following conditions. Let $n=(2k+1)(s+t)$ for some sufficiently large $k$ (as determined by Lemma \ref{No K22} with $p=s-1$).  Partition $U$ into $U=U_0\cup U_1\cup U_2$ and partition $V$ into $V=V_0\cup V_1\cup V_2$ where, $|U_1|=|V_2|=k(s+t)+\floor{\frac{t+1}{2}}$, $|V_1|=|U_2|=k(s+t)+\ceiling{\frac{t+1}{2}}$ and $|U_0|=|V_0|=s-1$.  Let $G[U_i,V_i]$ be complete for $i\in\{1,2\}$, $G[U_1,V_2]\cong P\left(k(s+t)+\floor{\frac{t+1}{2}},s-1\right)$ and $G[U_2,V_1]\cong P\left(k(s+t)+\ceiling{\frac{t+1}{2}},s-1\right)$.  Let $G[U_0,V_1\cup V_2]$ be complete, $G[V_0,U_1\cup U_2]$ be complete and $G[U_0,V_0]$ be empty.  Note that 
$$\delta(G)= \left\lbrace \begin{array}{ll}\frac{n+3s}{2}-\frac{3}{2}   & \text{ if } t \text{ is odd } \\
              \frac{n+3s}{2}-2 & \text{ if } t \text{ is even. } \end{array} \right. $$
Finally we reiterate the following properties of $G[U_1, V_2]$ and $G[U_2, V_1]$.  For $i=1,2$,
\begin{equation}\label{s-1}
\Delta(U_i,V_{3-i})=\Delta(V_i, U_{3-i})=s-1
\end{equation}
and
\begin{equation}\label{K_{2,2}}
G[U_i, V_{3-i}] \text{ is } K_{2,2}\text{-free}.
\end{equation}

For $i\in\{1,2\}$ and $A\in \{U_i,V_i\}$, let $A^D:=V_{3-i}$ if $A=U_i$ and let $A^D:=U_{3-i}$ if $A=V_i$.  We call $A^D$ the \emph{diagonal set of $A$}.  Let $A^N:=V_i$ if $A=U_i$ and $A^N:=U_i$ if $A=V_i$.  We call $A^N$ the \emph{non-diagonal set of $A$}. Finally, we let $A^M:=V_0$ if $A=U_i$ and $A^M:=U_0$ if $A=V_i$.  We call $A^M$ the \emph{opposite middle set of $A$}.

Suppose $K\cong K_{s,t}$ is a subgraph of $G$. We say $K$ is a \emph{crossing $K_{s,t}$} if $V(K)\cap(U_1\cup V_1)\neq \emptyset$ and $V(K)\cap(U_2\cup V_2)\neq \emptyset$.  Let $\mathcal{W}=\{U_1, U_2, V_1, V_2\}$.  

\begin{claim}\label{claim:properties}
If $K$ is a crossing $K_{s,t}$, then 
\begin{enumerate}
\item $V(K)$ must intersect some member of $\mathcal{W}$ in exactly one vertex, and

\item there is a unique $A_0\in \{U_0, V_0\}$ such that $V(K)\cap A_0\neq \emptyset$.
\end{enumerate}
Furthermore, if $|V(K)\cap A|=1$ for some $A\in \mathcal{W}$, then 
\begin{enumerate}[resume]
\item $V(K)\cap A^D\neq \emptyset$, and

\item either $|V(K)\cap A^N|\geq 2$ and $V(K)\cap (A^N)^D=\emptyset$, or $V(K)\cap A^N=\emptyset$ and $|V(K)\cap (A^N)^D|\geq 2$.

\end{enumerate}

\end{claim}

\begin{proof}
\begin{enumerate}
\item Suppose not. Then without loss of generality, suppose that $|V(K)\cap V_1|\geq 2$. By (\ref{K_{2,2}}) we have, $|V(K)\cap U_2|\leq 1$ and thus $V(K)\cap U_2=\emptyset$.  Since $K$ is crossing, we have $V(K)\cap V_2\neq \emptyset$ and thus $|V(K)\cap V_2|\geq 2$. By (\ref{K_{2,2}}) we have, $|V(K)\cap U_1|\leq 1$ and thus $V(K)\cap U_1=\emptyset$.  This is a contradiction, since $K\cong K_{s,t}$ and $|V(K)\cap U|\leq |U_0|=s-1$.

\item Suppose first that $V(K)\cap U_0=\emptyset=V(K)\cap V_0$.  By Claim \ref{claim:properties} (i), we can assume without loss of generality that $|V(K)\cap U_1|=1$.  Then either $|V(K)\cap U_2|=t-1$ or $|V(K)\cap U_2|=s-1$.  If $|V(K)\cap U_2|=t-1$, then by (\ref{s-1}) we must have $V(K)\cap V_1=\emptyset$ which implies $|V(K)\cap V_2|=s$, contradicting (\ref{s-1}).  If $|V(K)\cap U_2|=s-1$, then since $t\geq 2s+1$ we have $|V(K)\cap V_1|\geq s+1$ or $|V(K)\cap V_2|\geq s+1$, both of which contradict (\ref{s-1}).  Thus there exists $A_0\in \{U_0, V_0\}$ such that $V(K)\cap A_0\neq \emptyset$.  Finally since $G[U_0, V_0]$ is empty, $A_0$ must be unique.

\item Suppose that $V(K)\cap A^D=\emptyset$. Since $|V_0|=s-1$, we have $V(K)\cap A^N\neq \emptyset$ and since $K$ is crossing, we have $V(K)\cap (A^N)^D\neq \emptyset$.  Then by (\ref{s-1}), we have $|V(K)\cap A^N|, |V(K)\cap (A^N)^D|\leq s-1$.  Thus $|V(K)\cap U|\leq 2s-1$ and $|V(K)\cap V|\leq 2s-2$, contradicting the fact that $K\cong K_{s,t}$ and $t\geq 2s+1$.

\item We first show that it is not possible for either $|V(K)\cap A^N|=1$ or $|V(K)\cap (A^N)^D|=1$. If $|V(K)\cap A^N|=1$, then by (\ref{s-1}) and $|U_0|=|V_0|=s-1$, we have $|V(K)\cap U|, |V(K)\cap V|\leq 2s-1$, contradicting the fact that $K\cong K_{s,t}$ and $t\geq 2s+1$.  So suppose $|V(K)\cap (A^N)^D|=1$.  If $V(K)\cap U_0=\emptyset$, then $|V(K)\cap U|=2$ and since $t\geq 3$ we must have $s=2$. Then by (\ref{s-1}) we have $|V(K)\cap V|\leq 3$ contradicting the fact that $K\cong K_{s,t}$ and $t\geq 2s+1$.  If $V(K)\cap U_0\neq \emptyset$, then $V(K)\cap V_0=\emptyset$.  So $|V(K)\cap U|\leq s+1$ and by (\ref{s-1}), $|V(K)\cap V|\leq 2s-2$ contradicting the fact that $K\cong K_{s,t}$ and $t\geq 2s+1$.

Now suppose $V(K)\cap A^N\neq\emptyset$ and $V(K)\cap (A^N)^D\neq\emptyset$. Thus, by the previous paragraph we have $|V(K)\cap A^N|, |V(K)\cap (A^N)^D|\geq 2$, contradicting (\ref{K_{2,2}}).

So suppose that $V(K)\cap A^N=\emptyset=V(K)\cap (A^N)^D$.  Then it must be the case that $|V(K)\cap (A^N)^M|=s-1$ and consequently $|V(K)\cap A^D|=t$, contradicting (\ref{s-1}).

\end{enumerate}

\end{proof}

Let $A\in\mathcal{W}$.  We say $K$ is \emph{crossing from $A$} if either $|V(K)\cap A|=1$ and $|V(K)\cap A^D|\geq 2$, or $|V(K)\cap A|=1$, $|V(K)\cap A^D|=1$ and $V(K)\cap A^M\neq \emptyset$.  We say that a crossing $K_{s,t}$ from $A$ is \emph{Type 1} if $|V(K)\cap (A^N)^M|=s-1$, $|V(K)\cap A^N|=t-p$ and $|V(K)\cap A^D|=p$ for some $2\leq p\leq s-1$.  We say that a crossing $K_{s,t}$ from $A$ is \emph{Type 2} if $|V(K)\cap (A^N)^D|=t-1$, $|V(K)\cap A^M|=s-p$, and $|V(K)\cap A^D|=p$ for some $1\leq p\leq s-1$.

\begin{figure}[ht]
\centering
\subfloat[(Type 1 crossing $K_{s,t}$ from $A$)]{
\input{Type1.pstex_t}
\label{Type1}
}~~
\subfloat[(Type 2 crossing $K_{s,t}$ from $A$)]{
\input{Type2.pstex_t}
\label{Type2}
}
\label{typefig}
\caption[]{}
\end{figure}

\begin{claim}\label{types}
Every crossing $K_{s,t}$ is either Type 1 or Type 2.
\end{claim}

\begin{proof} (See Figure 1) Let $K$ be a crossing $K_{s,t}$ and without loss of generality suppose $K$ is crossing from $U_1$.  Let $p:=|V(K)\cap V_2|$. By Claim \ref{claim:properties} (iii) and (\ref{s-1}) we have $1\leq p\leq s-1$.  Suppose $K$ is not Type 1.  If $V(K)\cap U_2=\emptyset$, then $|V(K)\cap U_0|=s-1$ which implies $V(K)\cap V_0=\emptyset$ by Claim \ref{claim:properties} (ii).  Since $K$ is not Type 1, it must be the case that $|V(K)\cap V_2|=1$ and $|V(K)\cap V_1|=t-1$ in which case $K$ is not crossing from $U_1$, contradicting our assumption.  So we suppose that $V(K)\cap U_2\neq \emptyset$.  By Claim \ref{claim:properties} (iv) we have $|V(K)\cap U_2|\geq 2$ and $V(K)\cap V_1=\emptyset$, which implies that $|V(K)\cap V_0|=s-p$.  So by Claim \ref{claim:properties} (ii), we have $V(K)\cap U_0=\emptyset$ and thus $|V(K)\cap U_2|=t-1$, so $K$ is Type 2.

\end{proof}

Suppose for a contradiction that $G$ can be tiled with $K_{s,t}$.  Let $\mathcal{F}$ be a tiling of $G$ which minimizes the number of crossing $K_{s,t}$'s.

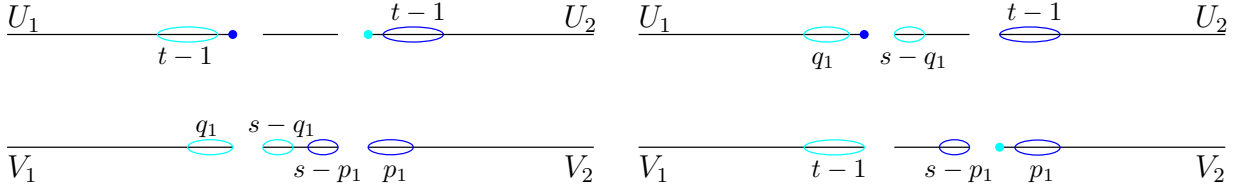
\begin{figure}[ht]
\centering
\subfloat{
\input{Case2a.pstex_t}
\label{fig22a}
}~~
\subfloat{
\input{Case2b.pstex_t}
\label{fig22b}
}
\label{forbfig}
\caption[]{Two cases in the proof of Claim \ref{forbidden}. 
}
\end{figure}

\begin{claim}\label{forbidden}
For $i=1,2$, if there is a crossing $K_{s,t}$ of Type $2$ from $U_i$ or $V_i$, then there is no crossing $K_{s,t}$ of Type $2$ from $U_{3-i}$ or $V_{3-i}$.
\end{claim}

\begin{proof}
Without loss of generality suppose $K^1$ is a crossing $K_{s,t}$ of Type $2$ from $U_1$.  Suppose that $K^2$ is a crossing $K_{s,t}$ of Type $2$ from $U_2$ (See Figure 2). For $i\in\{1,2\}$, let 
$$K^i_*:=G[U_i\cap (V(K^1)\cup V(K^2)), V(K^{3-i})\cap (V_0\cup V_i)].$$  
We have $K^1_*\cong K_{s,t} \cong K^2_*$, neither of $K^1_*,K^2_*$ are crossing, and $V(K^1)\cup V(K^2)=V(K^1_*)\cup V(K^2_*)$. Thus we obtain a tiling with fewer crossing $K_{s,t}$'s, contradicting the minimality of $\mathcal{F}$.

Now, suppose $K^1$ is a crossing $K_{s,t}$ of Type $2$ from $U_1$ and $K^2$ is a crossing $K_{s,t}$ of Type $2$ from $V_2$ (See Figure 2). Specify an element $L^1\in\mathcal{F}$, such that $V(L^1)\subseteq U_1\cup V_1$ and $|V(L^1)\cap V_1|=t$ and specify an element $L^2\in\mathcal{F}$, such that $V(L^2)\subseteq U_2\cup V_2$ and $|V(L^2)\cap U_2|=t$.  Choose arbitrary vertices $v'\in V(K^1)\cap V_0$ and $u'\in V(K^2)\cap U_0$.  We now define four subgraphs of $G$.  Let
\begin{align*}
K^1_*:&=G[V(L^1)\cap V_1, (V(K^1)\cup V(K^2))\cap ((U_1\cup U_0)\setminus\{u'\})],\\
L^1_*:&=G[V(L^1)\cap U_1, (V(K^2)\cap V_1)\cup\{v'\}],\\
K^2_*:&=G[V(L^2)\cap U_2, (V(K^1)\cup V(K^2))\cap ((V_2\cup V_0)\setminus\{v'\})], \text{ and}\\
L^2_*:&=G[V(L^2)\cap V_1, (V(K^1)\cap U_2)\cup\{u'\}]. 
\end{align*}

All of $K^1_*,K^2_*,L^1_*,L^2_*$ are isomorphic to $K_{s,t}$, none of $K^1_*,K^2_*,L^1_*,L^2_*$ are crossing, and $V(K^1_*)\cup V(K^2_*)\cup V(L^1_*)\cup V(L^2_*)=V(K^1)\cup V(K^2)\cup V(L^1)\cup V(L^2)$. Thus we obtain a tiling with fewer crossing $K_{s,t}$'s, contradicting the minimality of $\mathcal{F}$.
\end{proof}

For $i\in\{1,2\}$, let $\mathcal{F}_i$ be the set of all copies of $K_{s,t}$ in $\mathcal{F}$ which touch $U_i\cup V_i$. And let $U_i^*$ (resp.~$V_i^*$) be all the vertices in $U$ (resp.~$V$) which touch elements of $\mathcal{F}_i$.  Precisely, let $\mathcal{F}_i=\{K\in \mathcal{F}: V(K)\cap(U_i\cup V_i)\neq\emptyset\}$ for $i=1,2$, and let 
\begin{align*}
U_i^*=\left(\cup_{K\in\mathcal{F}_i}V(K)\right)\cap U ~~\text{ and }~~ V_i^*=\left(\cup_{K\in\mathcal{F}_i}V(K)\right)\cap V.
\end{align*}
Note that $U_i\subseteq U_i^*$ and $V_i\subseteq V_i^*$.  We will use the following claim to show that all of the remaining possible configurations of crossing $K_{s,t}$'s lead to contradictions.

\begin{claim}\label{claim:main}
For all $i\in\{1,2\}$, either $$\max\{|U_i^*|, |V_i^*|\}\geq k(s+t)+2t ~\text{ or }~ \min\{|U_i^*|, |V_i^*|\}\geq (k+1)(s+t).$$
\end{claim}

\begin{proof}
Suppose that $\max\{|U_i^*|, |V_i^*|\}< k(s+t)+2t$.  Then since $U_i\subseteq U_i^*$ and $V_i\subseteq V_i^*$, we have \begin{equation}\label{bounded} k(s+t)+s<|U_i^*|, |V_i^*|<k(s+t)+2t,\end{equation} and thus \begin{equation}\label{2t-s}||U_i^*|-|V_i^*||<2t-s.\end{equation}  By definition $G[U_i^*, V_i^*]$ can be tiled, thus there exists nonnegative integers $\ell, a, b$ such that $|U_i^*|=\ell(s+t)+as+bt$ and $|V_i^*|=\ell(s+t)+at+bs$.  By choosing $\ell$ to be maximal, we have $a=0$ or $b=0$.  
If $\ell\leq k-1$, then in order to satisfy the lower bound in \eqref{bounded} we must have $a\geq 3$ or $b\geq 3$.  Since $a=0$ or $b=0$, we have $||U_i^*|-|V_i^*||\geq 3t-3s\geq 2t-s$, which contradicts \eqref{2t-s}.  If $\ell=k$, then in order to satisfy the lower bound in \eqref{bounded}, we must have $a\geq 2$ or $b\geq 2$, but then we violate the upper bound.  So $\ell\geq k+1$ and we have $\min\{|U_i^*|, |V_i^*|\}\geq (k+1)(s+t)$.
\end{proof}

We will also use the following facts.  For $i=1,2$, we have
\begin{equation}\label{eq:upper}
|V_i\cup V_0|+s, |U_i\cup U_0|+s\leq k(s+t)+\frac{t+2}{2}+2s-1<(k+1)(s+t).
\end{equation}
which in particular implies
\begin{equation}\label{eq:upper2t}
|V_i\cup V_0|+t, |U_i\cup U_0|+t<k(s+t)+2t.
\end{equation}

\begin{figure}[ht]
\centering
\subfloat[(Case 1.0)]{
\input{Case1-0.pstex_t}
\label{fig10}
}~~
\subfloat[(Case 1.1.i)]{
\input{Case1ai.pstex_t}
\label{fig11a}
}\\
\subfloat[(Case 1.1.ii)]{
\input{Case1aii.pstex_t}
\label{fig11b}
}~~
\subfloat[(Case 1.2.i)]{
\input{Case1bi.pstex_t}
\label{fig12a}
}\\
\subfloat[(Case 1.2.ii)]{
\input{Case1bii.pstex_t}
\label{fig12b}
}
\label{fig1}
\caption[]{Case 1}
\end{figure}

Let $i\in\{1,2\}$ and let $X_i=\{K\in\mathcal{F}:K \text{ is crossing from } U_i \text{ and } K \text{ is Type } 2\}$ and $Y_i=\{K\in\mathcal{F}:K \text{ is crossing from } V_i \text{ and } K \text{ is Type } 2\}$.  Since $|U_0|=|V_0|=s-1$, Claim \ref{claim:properties} (ii) implies, \begin{equation}\label{X_iY_i}0\leq|X_i|,|Y_i|\leq s-1. \end{equation}

\noindent
\textbf{Case 0.} There are no crossing $K_{s,t}$'s.  So $|U_1^*|\leq |U_1\cup U_0|$ and $|V_1^*|\leq |V_1\cup V_0|$.  Then by (\ref{eq:upper}) we have $|U_1^*|, |V_1^*|<(k+1)(s+t)$, contradicting Claim \ref{claim:main}.

\noindent
\textbf{Case 1.} There is a crossing $K_{s,t}$ of Type $1$.  Without loss of generality, suppose $K^1$ is a crossing $K_{s,t}$ of Type $1$ from $U_1$ and let $p:=|V(K^1)\cap V_2|$.  Since $U_0\setminus V(K^1)=\emptyset$, there can be no other crossing $K_{s,t}$'s of Type 1 from $U_1$ or $U_2$ and no crossing $K_{s,t}$'s of Type 2 from $V_1$ or $V_2$. By Claim \ref{types}, we must only consider five subcases:

\textbf{Case 1.0.} $K^1$ is the only crossing $K_{s,t}$. So $|U_1^*|\leq |U_1\cup U_0|$ and $|V_1^*|\leq |V_1\cup V_0|+p< |V_1\cup V_0|+s$.  Then by (\ref{eq:upper}) we have $|U_1^*|, |V_1^*|<(k+1)(s+t)$, contradicting Claim \ref{claim:main}.

\textbf{Case 1.1.i.} There is a crossing $K_{s,t}$ of Type $1$ from $V_1$. Let $K^2$ be a crossing $K_{s,t}$ from $V_1$ and let $q:=|V(K^2)\cap U_2|$.  Since $V_0\setminus V(K^2)=\emptyset$, $K^1$ and $K^2$ are the only crossing $K_{s,t}$'s.  So $|U_1^*|\leq |U_1\cup U_0|+q< |U_1\cup U_0|+s$ and $|V_1^*|\leq |V_1\cup V_0|+p< |V_1\cup V_0|+s$.  Then by (\ref{eq:upper}) we have, $|U_1^*|, |V_1^*|<(k+1)(s+t)$, contradicting Claim \ref{claim:main}.

\textbf{Case 1.1.ii.} There is a crossing $K_{s,t}$ of Type $1$ from $V_2$. Let $K^2$ be a crossing $K_{s,t}$ from $V_2$ and let $q:=|V(K^2)\cap U_1|$.  Since $V_0\setminus V(K^2)=\emptyset$, $K^1$ and $K^2$ are the only crossing $K_{s,t}$'s.  So $|V_1^*|\leq |V_1\cup V_0|+p+1\leq |V_1\cup V_0|+s$ and $|U_1^*|\leq |U_1\cup U_0|+t-q<|U_1\cup U_0|+t$.  Then by (\ref{eq:upper}) and (\ref{eq:upper2t}) we have $|V_1^*|<(k+1)(s+t)$ and $|U_1^*|<k(s+t)+2t$, contradicting Claim \ref{claim:main}.

\textbf{Case 1.2.i.} $1\leq |X_1|$. By Claim \ref{forbidden}, since there exists a crossing $K_{s,t}$ of Type $2$ from $U_1$, there can be no crossing $K_{s,t}$'s of Type $2$ from $U_2$.  So $|U_2^*|\leq |U_2\cup U_0|+|X_1|+1\leq |U_2\cup U_0|+s$ and $|V_2^*|\leq |V_2\cup V_0|+t-p<|V_2\cup V_0|+t$.  Then by (\ref{eq:upper}) and (\ref{eq:upper2t}) we have $|U_2^*|<(k+1)(s+t)$ and $|V_2^*|<k(s+t)+2t$, contradicting Claim \ref{claim:main}.

\textbf{Case 1.2.ii.} $1\leq |X_2|$. By Claim \ref{forbidden}, since there exists a crossing $K_{s,t}$ of Type $2$ from $U_2$, then there can be no crossing $K_{s,t}$'s of Type $2$ from $U_1$.  So $|U_1^*|\leq |U_1\cup U_0|+|X_2|<|U_1\cup U_0|+s$ and $|V_1^*|\leq |V_1\cup V_0|+p<|V_1\cup V_0|+s$.  Then by (\ref{eq:upper}) we have $|U_1^*|, |V_1^*|<(k+1)(s+t)$, contradicting Claim \ref{claim:main}.

\begin{figure}[ht]
\centering
\subfloat{
\input{Case2d.pstex_t}
\label{fig20}
}~~
\subfloat{
\input{Case2c.pstex_t}
\label{fig22}
}
\label{fig2}
\caption[]{Case 2
}
\end{figure}

\noindent
\textbf{Case 2.} There are no crossing $K_{s,t}$'s of Type $1$.  By Claim  \ref{types}, there can only be crossing $K_{s,t}$'s of Type $2$.  Without loss of generality suppose that $1\leq |X_1|$.  Then there can be no crossing $K_{s,t}$ of Type $2$ from $U_2$ or $V_2$. So $|U_2^*|\leq |U_2\cup U_0|+|X_1|<|U_2\cup U_0|+s$ and $|V_2^*|\leq |V_2\cup V_0|+|Y_1|<|V_2\cup V_0|+s$.  Then by (\ref{eq:upper}) we have $|U_2^*|, |V_2^*|<(k+1)(s+t)$, contradicting Claim \ref{claim:main}.

\end{proof}

\subsection*{Acknowledgements}

We thank the very careful referees for their suggestions which improved the presentation of this paper.

\end{document}

%% file: Type1.pstex_t
\begin{picture}(0,0)%
\includegraphics{Type1.pstex}%
\end{picture}%
\setlength{\unitlength}{4144sp}%
\begingroup\makeatletter\ifx\SetFigFontNFSS\undefined%
\gdef\SetFigFontNFSS#1#2#3#4#5{%
  \reset@font\fontsize{#1}{#2pt}%
  \fontfamily{#3}\fontseries{#4}\fontshape{#5}%
  \selectfont}%
\fi\endgroup%
\begin{picture}(3537,1132)(436,-1880)
\put(451,-916){\makebox(0,0)[lb]{\smash{{\SetFigFontNFSS{12}{14.4}{\familydefault}{\mddefault}{\updefault}{\color[rgb]{0,0,0}$A$}%
}}}}
\put(1441,-1816){\makebox(0,0)[lb]{\smash{{\SetFigFontNFSS{10}{12.0}{\familydefault}{\mddefault}{\updefault}{\color[rgb]{0,0,0}$t-p$}%
}}}}
\put(2071,-871){\makebox(0,0)[lb]{\smash{{\SetFigFontNFSS{10}{12.0}{\familydefault}{\mddefault}{\updefault}{\color[rgb]{0,0,0}$s-1$}%
}}}}
\put(451,-1816){\makebox(0,0)[lb]{\smash{{\SetFigFontNFSS{12}{14.4}{\familydefault}{\mddefault}{\updefault}{\color[rgb]{0,0,0}$A^V$}%
}}}}
\put(3736,-1816){\makebox(0,0)[lb]{\smash{{\SetFigFontNFSS{12}{14.4}{\familydefault}{\mddefault}{\updefault}{\color[rgb]{0,0,0}$A^D$}%
}}}}
\put(2431,-1816){\makebox(0,0)[lb]{\smash{{\SetFigFontNFSS{10}{12.0}{\familydefault}{\mddefault}{\updefault}{\color[rgb]{0,0,0}$2\leq p\leq s-1$}%
}}}}
\end{picture}%

%% file: Type2.pstex_t
\begin{picture}(0,0)%
\includegraphics{Type2.pstex}%
\end{picture}%
\setlength{\unitlength}{4144sp}%
\begingroup\makeatletter\ifx\SetFigFontNFSS\undefined%
\gdef\SetFigFontNFSS#1#2#3#4#5{%
  \reset@font\fontsize{#1}{#2pt}%
  \fontfamily{#3}\fontseries{#4}\fontshape{#5}%
  \selectfont}%
\fi\endgroup%
\begin{picture}(3537,1132)(436,-1880)
\put(2701,-871){\makebox(0,0)[lb]{\smash{{\SetFigFontNFSS{10}{12.0}{\familydefault}{\mddefault}{\updefault}{\color[rgb]{0,0,0}$t-1$}%
}}}}
\put(1891,-1816){\makebox(0,0)[lb]{\smash{{\SetFigFontNFSS{10}{12.0}{\familydefault}{\mddefault}{\updefault}{\color[rgb]{0,0,0}$s-p$}%
}}}}
\put(2431,-1816){\makebox(0,0)[lb]{\smash{{\SetFigFontNFSS{10}{12.0}{\familydefault}{\mddefault}{\updefault}{\color[rgb]{0,0,0}$1\leq p\leq s-1$}%
}}}}
\put(451,-916){\makebox(0,0)[lb]{\smash{{\SetFigFontNFSS{12}{14.4}{\familydefault}{\mddefault}{\updefault}{\color[rgb]{0,0,0}$A$}%
}}}}
\put(3736,-1816){\makebox(0,0)[lb]{\smash{{\SetFigFontNFSS{12}{14.4}{\familydefault}{\mddefault}{\updefault}{\color[rgb]{0,0,0}$A^D$}%
}}}}
\put(3556,-916){\makebox(0,0)[lb]{\smash{{\SetFigFontNFSS{12}{14.4}{\familydefault}{\mddefault}{\updefault}{\color[rgb]{0,0,0}$(A^V)^D$}%
}}}}
\end{picture}%

%% file: Case2a.pstex_t
\begin{picture}(0,0)%
\includegraphics{Case2a.pstex}%
\end{picture}%
\setlength{\unitlength}{4144sp}%
\begingroup\makeatletter\ifx\SetFigFontNFSS\undefined%
\gdef\SetFigFontNFSS#1#2#3#4#5{%
  \reset@font\fontsize{#1}{#2pt}%
  \fontfamily{#3}\fontseries{#4}\fontshape{#5}%
  \selectfont}%
\fi\endgroup%
\begin{picture}(3537,1137)(436,-1885)
\put(1576,-1546){\makebox(0,0)[lb]{\smash{{\SetFigFontNFSS{10}{12.0}{\familydefault}{\mddefault}{\updefault}{\color[rgb]{0,0,0}$q_1$}%
}}}}
\put(1891,-1546){\makebox(0,0)[lb]{\smash{{\SetFigFontNFSS{10}{12.0}{\familydefault}{\mddefault}{\updefault}{\color[rgb]{0,0,0}$s-q_1$}%
}}}}
\put(2746,-871){\makebox(0,0)[lb]{\smash{{\SetFigFontNFSS{10}{12.0}{\familydefault}{\mddefault}{\updefault}{\color[rgb]{0,0,0}$t-1$}%
}}}}
\put(2161,-1816){\makebox(0,0)[lb]{\smash{{\SetFigFontNFSS{10}{12.0}{\familydefault}{\mddefault}{\updefault}{\color[rgb]{0,0,0}$s-p_1$}%
}}}}
\put(2701,-1816){\makebox(0,0)[lb]{\smash{{\SetFigFontNFSS{10}{12.0}{\familydefault}{\mddefault}{\updefault}{\color[rgb]{0,0,0}$p_1$}%
}}}}
\put(1351,-1141){\makebox(0,0)[lb]{\smash{{\SetFigFontNFSS{10}{12.0}{\familydefault}{\mddefault}{\updefault}{\color[rgb]{0,0,0}$t-1$}%
}}}}
\put(451,-916){\makebox(0,0)[lb]{\smash{{\SetFigFontNFSS{12}{14.4}{\familydefault}{\mddefault}{\updefault}{\color[rgb]{0,0,0}$U_1$}%
}}}}
\put(451,-1816){\makebox(0,0)[lb]{\smash{{\SetFigFontNFSS{12}{14.4}{\familydefault}{\mddefault}{\updefault}{\color[rgb]{0,0,0}$V_1$}%
}}}}
\put(3781,-916){\makebox(0,0)[lb]{\smash{{\SetFigFontNFSS{12}{14.4}{\familydefault}{\mddefault}{\updefault}{\color[rgb]{0,0,0}$U_2$}%
}}}}
\put(3781,-1816){\makebox(0,0)[lb]{\smash{{\SetFigFontNFSS{12}{14.4}{\familydefault}{\mddefault}{\updefault}{\color[rgb]{0,0,0}$V_2$}%
}}}}
\end{picture}%

%% file: Case2b.pstex_t
\begin{picture}(0,0)%
\includegraphics{Case2b.pstex}%
\end{picture}%
\setlength{\unitlength}{4144sp}%
\begingroup\makeatletter\ifx\SetFigFontNFSS\undefined%
\gdef\SetFigFontNFSS#1#2#3#4#5{%
  \reset@font\fontsize{#1}{#2pt}%
  \fontfamily{#3}\fontseries{#4}\fontshape{#5}%
  \selectfont}%
\fi\endgroup%
\begin{picture}(3537,1137)(436,-1885)
\put(2161,-1816){\makebox(0,0)[lb]{\smash{{\SetFigFontNFSS{10}{12.0}{\familydefault}{\mddefault}{\updefault}{\color[rgb]{0,0,0}$s-p_1$}%
}}}}
\put(1891,-1141){\makebox(0,0)[lb]{\smash{{\SetFigFontNFSS{10}{12.0}{\familydefault}{\mddefault}{\updefault}{\color[rgb]{0,0,0}$s-q_1$}%
}}}}
\put(1486,-1816){\makebox(0,0)[lb]{\smash{{\SetFigFontNFSS{10}{12.0}{\familydefault}{\mddefault}{\updefault}{\color[rgb]{0,0,0}$t-1$}%
}}}}
\put(2656,-871){\makebox(0,0)[lb]{\smash{{\SetFigFontNFSS{10}{12.0}{\familydefault}{\mddefault}{\updefault}{\color[rgb]{0,0,0}$t-1$}%
}}}}
\put(2791,-1816){\makebox(0,0)[lb]{\smash{{\SetFigFontNFSS{10}{12.0}{\familydefault}{\mddefault}{\updefault}{\color[rgb]{0,0,0}$p_1$}%
}}}}
\put(1486,-1141){\makebox(0,0)[lb]{\smash{{\SetFigFontNFSS{10}{12.0}{\familydefault}{\mddefault}{\updefault}{\color[rgb]{0,0,0}$q_1$}%
}}}}
\put(451,-916){\makebox(0,0)[lb]{\smash{{\SetFigFontNFSS{12}{14.4}{\familydefault}{\mddefault}{\updefault}{\color[rgb]{0,0,0}$U_1$}%
}}}}
\put(451,-1816){\makebox(0,0)[lb]{\smash{{\SetFigFontNFSS{12}{14.4}{\familydefault}{\mddefault}{\updefault}{\color[rgb]{0,0,0}$V_1$}%
}}}}
\put(3781,-916){\makebox(0,0)[lb]{\smash{{\SetFigFontNFSS{12}{14.4}{\familydefault}{\mddefault}{\updefault}{\color[rgb]{0,0,0}$U_2$}%
}}}}
\put(3781,-1816){\makebox(0,0)[lb]{\smash{{\SetFigFontNFSS{12}{14.4}{\familydefault}{\mddefault}{\updefault}{\color[rgb]{0,0,0}$V_2$}%
}}}}
\end{picture}%

%% file: Case1-0.pstex_t
\begin{picture}(0,0)%
\includegraphics{Case1-0.pstex}%
\end{picture}%
\setlength{\unitlength}{4144sp}%
\begingroup\makeatletter\ifx\SetFigFontNFSS\undefined%
\gdef\SetFigFontNFSS#1#2#3#4#5{%
  \reset@font\fontsize{#1}{#2pt}%
  \fontfamily{#3}\fontseries{#4}\fontshape{#5}%
  \selectfont}%
\fi\endgroup%
\begin{picture}(3537,1137)(436,-1885)
\put(1441,-1816){\makebox(0,0)[lb]{\smash{{\SetFigFontNFSS{10}{12.0}{\familydefault}{\mddefault}{\updefault}{\color[rgb]{0,0,0}$t-p$}%
}}}}
\put(2071,-871){\makebox(0,0)[lb]{\smash{{\SetFigFontNFSS{10}{12.0}{\familydefault}{\mddefault}{\updefault}{\color[rgb]{0,0,0}$s-1$}%
}}}}
\put(2611,-1816){\makebox(0,0)[lb]{\smash{{\SetFigFontNFSS{10}{12.0}{\familydefault}{\mddefault}{\updefault}{\color[rgb]{0,0,0}$p$}%
}}}}
\put(451,-916){\makebox(0,0)[lb]{\smash{{\SetFigFontNFSS{12}{14.4}{\familydefault}{\mddefault}{\updefault}{\color[rgb]{0,0,0}$U_1$}%
}}}}
\put(451,-1816){\makebox(0,0)[lb]{\smash{{\SetFigFontNFSS{12}{14.4}{\familydefault}{\mddefault}{\updefault}{\color[rgb]{0,0,0}$V_1$}%
}}}}
\put(3781,-916){\makebox(0,0)[lb]{\smash{{\SetFigFontNFSS{12}{14.4}{\familydefault}{\mddefault}{\updefault}{\color[rgb]{0,0,0}$U_2$}%
}}}}
\put(3781,-1816){\makebox(0,0)[lb]{\smash{{\SetFigFontNFSS{12}{14.4}{\familydefault}{\mddefault}{\updefault}{\color[rgb]{0,0,0}$V_2$}%
}}}}
\end{picture}%

%% file: Case1ai.pstex_t
\begin{picture}(0,0)%
\includegraphics{Case1ai.pstex}%
\end{picture}%
\setlength{\unitlength}{4144sp}%
\begingroup\makeatletter\ifx\SetFigFontNFSS\undefined%
\gdef\SetFigFontNFSS#1#2#3#4#5{%
  \reset@font\fontsize{#1}{#2pt}%
  \fontfamily{#3}\fontseries{#4}\fontshape{#5}%
  \selectfont}%
\fi\endgroup%
\begin{picture}(3537,1137)(436,-1885)
\put(2071,-871){\makebox(0,0)[lb]{\smash{{\SetFigFontNFSS{10}{12.0}{\familydefault}{\mddefault}{\updefault}{\color[rgb]{0,0,0}$s-1$}%
}}}}
\put(2611,-1816){\makebox(0,0)[lb]{\smash{{\SetFigFontNFSS{10}{12.0}{\familydefault}{\mddefault}{\updefault}{\color[rgb]{0,0,0}$p$}%
}}}}
\put(1351,-1816){\makebox(0,0)[lb]{\smash{{\SetFigFontNFSS{10}{12.0}{\familydefault}{\mddefault}{\updefault}{\color[rgb]{0,0,0}$t-p$}%
}}}}
\put(2071,-1546){\makebox(0,0)[lb]{\smash{{\SetFigFontNFSS{10}{12.0}{\familydefault}{\mddefault}{\updefault}{\color[rgb]{0,0,0}$s-1$}%
}}}}
\put(1351,-1141){\makebox(0,0)[lb]{\smash{{\SetFigFontNFSS{10}{12.0}{\familydefault}{\mddefault}{\updefault}{\color[rgb]{0,0,0}$t-q$}%
}}}}
\put(2611,-1141){\makebox(0,0)[lb]{\smash{{\SetFigFontNFSS{10}{12.0}{\familydefault}{\mddefault}{\updefault}{\color[rgb]{0,0,0}$q$}%
}}}}
\put(451,-916){\makebox(0,0)[lb]{\smash{{\SetFigFontNFSS{12}{14.4}{\familydefault}{\mddefault}{\updefault}{\color[rgb]{0,0,0}$U_1$}%
}}}}
\put(451,-1816){\makebox(0,0)[lb]{\smash{{\SetFigFontNFSS{12}{14.4}{\familydefault}{\mddefault}{\updefault}{\color[rgb]{0,0,0}$V_1$}%
}}}}
\put(3781,-916){\makebox(0,0)[lb]{\smash{{\SetFigFontNFSS{12}{14.4}{\familydefault}{\mddefault}{\updefault}{\color[rgb]{0,0,0}$U_2$}%
}}}}
\put(3781,-1816){\makebox(0,0)[lb]{\smash{{\SetFigFontNFSS{12}{14.4}{\familydefault}{\mddefault}{\updefault}{\color[rgb]{0,0,0}$V_2$}%
}}}}
\end{picture}%

%% file: Case1aii.pstex_t
\begin{picture}(0,0)%
\includegraphics{Case1aii.pstex}%
\end{picture}%
\setlength{\unitlength}{4144sp}%
\begingroup\makeatletter\ifx\SetFigFontNFSS\undefined%
\gdef\SetFigFontNFSS#1#2#3#4#5{%
  \reset@font\fontsize{#1}{#2pt}%
  \fontfamily{#3}\fontseries{#4}\fontshape{#5}%
  \selectfont}%
\fi\endgroup%
\begin{picture}(3537,1137)(436,-1885)
\put(2071,-871){\makebox(0,0)[lb]{\smash{{\SetFigFontNFSS{10}{12.0}{\familydefault}{\mddefault}{\updefault}{\color[rgb]{0,0,0}$s-1$}%
}}}}
\put(2071,-1546){\makebox(0,0)[lb]{\smash{{\SetFigFontNFSS{10}{12.0}{\familydefault}{\mddefault}{\updefault}{\color[rgb]{0,0,0}$s-1$}%
}}}}
\put(2746,-1816){\makebox(0,0)[lb]{\smash{{\SetFigFontNFSS{10}{12.0}{\familydefault}{\mddefault}{\updefault}{\color[rgb]{0,0,0}$p$}%
}}}}
\put(1576,-1141){\makebox(0,0)[lb]{\smash{{\SetFigFontNFSS{10}{12.0}{\familydefault}{\mddefault}{\updefault}{\color[rgb]{0,0,0}$q$}%
}}}}
\put(2701,-1141){\makebox(0,0)[lb]{\smash{{\SetFigFontNFSS{10}{12.0}{\familydefault}{\mddefault}{\updefault}{\color[rgb]{0,0,0}$t-q$}%
}}}}
\put(1441,-1816){\makebox(0,0)[lb]{\smash{{\SetFigFontNFSS{10}{12.0}{\familydefault}{\mddefault}{\updefault}{\color[rgb]{0,0,0}$t-p$}%
}}}}
\put(451,-916){\makebox(0,0)[lb]{\smash{{\SetFigFontNFSS{12}{14.4}{\familydefault}{\mddefault}{\updefault}{\color[rgb]{0,0,0}$U_1$}%
}}}}
\put(451,-1816){\makebox(0,0)[lb]{\smash{{\SetFigFontNFSS{12}{14.4}{\familydefault}{\mddefault}{\updefault}{\color[rgb]{0,0,0}$V_1$}%
}}}}
\put(3781,-916){\makebox(0,0)[lb]{\smash{{\SetFigFontNFSS{12}{14.4}{\familydefault}{\mddefault}{\updefault}{\color[rgb]{0,0,0}$U_2$}%
}}}}
\put(3781,-1816){\makebox(0,0)[lb]{\smash{{\SetFigFontNFSS{12}{14.4}{\familydefault}{\mddefault}{\updefault}{\color[rgb]{0,0,0}$V_2$}%
}}}}
\end{picture}%

%% file: Case1bi.pstex_t
\begin{picture}(0,0)%
\includegraphics{Case1bi.pstex}%
\end{picture}%
\setlength{\unitlength}{4144sp}%
\begingroup\makeatletter\ifx\SetFigFontNFSS\undefined%
\gdef\SetFigFontNFSS#1#2#3#4#5{%
  \reset@font\fontsize{#1}{#2pt}%
  \fontfamily{#3}\fontseries{#4}\fontshape{#5}%
  \selectfont}%
\fi\endgroup%
\begin{picture}(3537,1137)(436,-1885)
\put(2071,-871){\makebox(0,0)[lb]{\smash{{\SetFigFontNFSS{10}{12.0}{\familydefault}{\mddefault}{\updefault}{\color[rgb]{0,0,0}$s-1$}%
}}}}
\put(2701,-1141){\makebox(0,0)[lb]{\smash{{\SetFigFontNFSS{10}{12.0}{\familydefault}{\mddefault}{\updefault}{\color[rgb]{0,0,0}$t-1$}%
}}}}
\put(1441,-1816){\makebox(0,0)[lb]{\smash{{\SetFigFontNFSS{10}{12.0}{\familydefault}{\mddefault}{\updefault}{\color[rgb]{0,0,0}$t-p$}%
}}}}
\put(2656,-1816){\makebox(0,0)[lb]{\smash{{\SetFigFontNFSS{10}{12.0}{\familydefault}{\mddefault}{\updefault}{\color[rgb]{0,0,0}$p$}%
}}}}
\put(2926,-1546){\makebox(0,0)[lb]{\smash{{\SetFigFontNFSS{10}{12.0}{\familydefault}{\mddefault}{\updefault}{\color[rgb]{0,0,0}$q_1$}%
}}}}
\put(2161,-1546){\makebox(0,0)[lb]{\smash{{\SetFigFontNFSS{10}{12.0}{\familydefault}{\mddefault}{\updefault}{\color[rgb]{0,0,0}$s-q_1$}%
}}}}
\put(451,-916){\makebox(0,0)[lb]{\smash{{\SetFigFontNFSS{12}{14.4}{\familydefault}{\mddefault}{\updefault}{\color[rgb]{0,0,0}$U_1$}%
}}}}
\put(451,-1816){\makebox(0,0)[lb]{\smash{{\SetFigFontNFSS{12}{14.4}{\familydefault}{\mddefault}{\updefault}{\color[rgb]{0,0,0}$V_1$}%
}}}}
\put(3781,-916){\makebox(0,0)[lb]{\smash{{\SetFigFontNFSS{12}{14.4}{\familydefault}{\mddefault}{\updefault}{\color[rgb]{0,0,0}$U_2$}%
}}}}
\put(3781,-1816){\makebox(0,0)[lb]{\smash{{\SetFigFontNFSS{12}{14.4}{\familydefault}{\mddefault}{\updefault}{\color[rgb]{0,0,0}$V_2$}%
}}}}
\end{picture}%

%% file: Case1bii.pstex_t
\begin{picture}(0,0)%
\includegraphics{Case1bii.pstex}%
\end{picture}%
\setlength{\unitlength}{4144sp}%
\begingroup\makeatletter\ifx\SetFigFontNFSS\undefined%
\gdef\SetFigFontNFSS#1#2#3#4#5{%
  \reset@font\fontsize{#1}{#2pt}%
  \fontfamily{#3}\fontseries{#4}\fontshape{#5}%
  \selectfont}%
\fi\endgroup%
\begin{picture}(3537,1137)(436,-1885)
\put(2071,-871){\makebox(0,0)[lb]{\smash{{\SetFigFontNFSS{10}{12.0}{\familydefault}{\mddefault}{\updefault}{\color[rgb]{0,0,0}$s-1$}%
}}}}
\put(2656,-1816){\makebox(0,0)[lb]{\smash{{\SetFigFontNFSS{10}{12.0}{\familydefault}{\mddefault}{\updefault}{\color[rgb]{0,0,0}$p$}%
}}}}
\put(1981,-1546){\makebox(0,0)[lb]{\smash{{\SetFigFontNFSS{10}{12.0}{\familydefault}{\mddefault}{\updefault}{\color[rgb]{0,0,0}$s-q_1$}%
}}}}
\put(1171,-1816){\makebox(0,0)[lb]{\smash{{\SetFigFontNFSS{10}{12.0}{\familydefault}{\mddefault}{\updefault}{\color[rgb]{0,0,0}$t-p$}%
}}}}
\put(1351,-1141){\makebox(0,0)[lb]{\smash{{\SetFigFontNFSS{10}{12.0}{\familydefault}{\mddefault}{\updefault}{\color[rgb]{0,0,0}$t-1$}%
}}}}
\put(1666,-1546){\makebox(0,0)[lb]{\smash{{\SetFigFontNFSS{10}{12.0}{\familydefault}{\mddefault}{\updefault}{\color[rgb]{0,0,0}$q_1$}%
}}}}
\put(451,-916){\makebox(0,0)[lb]{\smash{{\SetFigFontNFSS{12}{14.4}{\familydefault}{\mddefault}{\updefault}{\color[rgb]{0,0,0}$U_1$}%
}}}}
\put(451,-1816){\makebox(0,0)[lb]{\smash{{\SetFigFontNFSS{12}{14.4}{\familydefault}{\mddefault}{\updefault}{\color[rgb]{0,0,0}$V_1$}%
}}}}
\put(3781,-916){\makebox(0,0)[lb]{\smash{{\SetFigFontNFSS{12}{14.4}{\familydefault}{\mddefault}{\updefault}{\color[rgb]{0,0,0}$U_2$}%
}}}}
\put(3781,-1816){\makebox(0,0)[lb]{\smash{{\SetFigFontNFSS{12}{14.4}{\familydefault}{\mddefault}{\updefault}{\color[rgb]{0,0,0}$V_2$}%
}}}}
\end{picture}%

%% file: Case2d.pstex_t
\begin{picture}(0,0)%
\includegraphics{Case2d.pstex}%
\end{picture}%
\setlength{\unitlength}{4144sp}%
\begingroup\makeatletter\ifx\SetFigFontNFSS\undefined%
\gdef\SetFigFontNFSS#1#2#3#4#5{%
  \reset@font\fontsize{#1}{#2pt}%
  \fontfamily{#3}\fontseries{#4}\fontshape{#5}%
  \selectfont}%
\fi\endgroup%
\begin{picture}(3537,1137)(436,-1885)
\put(2161,-1816){\makebox(0,0)[lb]{\smash{{\SetFigFontNFSS{10}{12.0}{\familydefault}{\mddefault}{\updefault}{\color[rgb]{0,0,0}$s-p_1$}%
}}}}
\put(2656,-871){\makebox(0,0)[lb]{\smash{{\SetFigFontNFSS{10}{12.0}{\familydefault}{\mddefault}{\updefault}{\color[rgb]{0,0,0}$t-1$}%
}}}}
\put(2701,-1816){\makebox(0,0)[lb]{\smash{{\SetFigFontNFSS{10}{12.0}{\familydefault}{\mddefault}{\updefault}{\color[rgb]{0,0,0}$p_1$}%
}}}}
\put(451,-916){\makebox(0,0)[lb]{\smash{{\SetFigFontNFSS{12}{14.4}{\familydefault}{\mddefault}{\updefault}{\color[rgb]{0,0,0}$U_1$}%
}}}}
\put(451,-1816){\makebox(0,0)[lb]{\smash{{\SetFigFontNFSS{12}{14.4}{\familydefault}{\mddefault}{\updefault}{\color[rgb]{0,0,0}$V_1$}%
}}}}
\put(3781,-916){\makebox(0,0)[lb]{\smash{{\SetFigFontNFSS{12}{14.4}{\familydefault}{\mddefault}{\updefault}{\color[rgb]{0,0,0}$U_2$}%
}}}}
\put(3781,-1816){\makebox(0,0)[lb]{\smash{{\SetFigFontNFSS{12}{14.4}{\familydefault}{\mddefault}{\updefault}{\color[rgb]{0,0,0}$V_2$}%
}}}}
\end{picture}%

%% file: Case2c.pstex_t
\begin{picture}(0,0)%
\includegraphics{Case2c.pstex}%
\end{picture}%
\setlength{\unitlength}{4144sp}%
\begingroup\makeatletter\ifx\SetFigFontNFSS\undefined%
\gdef\SetFigFontNFSS#1#2#3#4#5{%
  \reset@font\fontsize{#1}{#2pt}%
  \fontfamily{#3}\fontseries{#4}\fontshape{#5}%
  \selectfont}%
\fi\endgroup%
\begin{picture}(3537,1137)(436,-1885)
\put(2161,-1816){\makebox(0,0)[lb]{\smash{{\SetFigFontNFSS{10}{12.0}{\familydefault}{\mddefault}{\updefault}{\color[rgb]{0,0,0}$s-p_1$}%
}}}}
\put(2656,-871){\makebox(0,0)[lb]{\smash{{\SetFigFontNFSS{10}{12.0}{\familydefault}{\mddefault}{\updefault}{\color[rgb]{0,0,0}$t-1$}%
}}}}
\put(3061,-1816){\makebox(0,0)[lb]{\smash{{\SetFigFontNFSS{10}{12.0}{\familydefault}{\mddefault}{\updefault}{\color[rgb]{0,0,0}$t-1$}%
}}}}
\put(2701,-1816){\makebox(0,0)[lb]{\smash{{\SetFigFontNFSS{10}{12.0}{\familydefault}{\mddefault}{\updefault}{\color[rgb]{0,0,0}$p_1$}%
}}}}
\put(2161,-1141){\makebox(0,0)[lb]{\smash{{\SetFigFontNFSS{10}{12.0}{\familydefault}{\mddefault}{\updefault}{\color[rgb]{0,0,0}$s-q_1$}%
}}}}
\put(3106,-1141){\makebox(0,0)[lb]{\smash{{\SetFigFontNFSS{10}{12.0}{\familydefault}{\mddefault}{\updefault}{\color[rgb]{0,0,0}$q_1$}%
}}}}
\put(451,-916){\makebox(0,0)[lb]{\smash{{\SetFigFontNFSS{12}{14.4}{\familydefault}{\mddefault}{\updefault}{\color[rgb]{0,0,0}$U_1$}%
}}}}
\put(451,-1816){\makebox(0,0)[lb]{\smash{{\SetFigFontNFSS{12}{14.4}{\familydefault}{\mddefault}{\updefault}{\color[rgb]{0,0,0}$V_1$}%
}}}}
\put(3781,-1816){\makebox(0,0)[lb]{\smash{{\SetFigFontNFSS{12}{14.4}{\familydefault}{\mddefault}{\updefault}{\color[rgb]{0,0,0}$V_2$}%
}}}}
\put(3781,-916){\makebox(0,0)[lb]{\smash{{\SetFigFontNFSS{12}{14.4}{\familydefault}{\mddefault}{\updefault}{\color[rgb]{0,0,0}$U_2$}%
}}}}
\end{picture}%